\documentclass[11pt]{article}
\usepackage[a4paper,margin=24mm]{geometry}
\usepackage[T1]{fontenc}
\usepackage{lmodern,microtype,amsmath,amssymb,amsthm,mathtools,booktabs,mathrsfs}
\usepackage{placeins}
\usepackage[labelsep=period]{caption}
\usepackage[colorlinks=true,linkcolor=blue,citecolor=blue,urlcolor=blue]{hyperref}
\hypersetup{pdftitle={Real-rootedness of Kazhdan--Lusztig polynomials of sparse paving matroids},pdfauthor={Philip B. Zhang}}
\numberwithin{equation}{section}
\newtheorem{theorem}{Theorem}[section]
\newtheorem{lemma}[theorem]{Lemma}
\newtheorem{proposition}[theorem]{Proposition}

\theoremstyle{remark}
\newcommand{\varthetax}{\vartheta_x}

\title{Real-rootedness of Kazhdan--Lusztig polynomials of sparse paving matroids}
\author{Philip B.~Zhang\\[0.4em]
\small College of Mathematical Sciences \& Institute of Mathematics and Interdisciplinary Sciences,\\
\small Tianjin Normal University, Tianjin 300387, P.R.~China\\[0.2em]
\small\texttt{zhang@tjnu.edu.cn}}
\date{}
\begin{document}
\maketitle
\begin{abstract}
We prove that every nonconstant Kazhdan--Lusztig polynomial of a sparse
paving matroid has only simple negative zeros. For each rank $d\ge5$,
the Kazhdan--Lusztig polynomial of the uniform matroid of rank $d-2$
on $d$ elements strictly interlaces those of all sparse paving matroids
of rank $d$ and corank at least two.
\end{abstract}



\section{Introduction}

Elias, Proudfoot, and Wakefield \cite{EPW} introduced the
Kazhdan--Lusztig polynomial $P_M(x)$ of a matroid $M$ and conjectured
that its coefficients are nonnegative and form a log-concave sequence.
They proved nonnegativity for representable matroids by interpreting the coefficients
as dimensions of intersection cohomology groups. Braden, Huh, Matherne, Proudfoot,
and Wang \cite{BHMPW} proved the nonnegativity conjecture for all matroids
using singular Hodge theory. Gedeon, Proudfoot, and Young
\cite{GPY} proposed the stronger conjecture that the
zeros are all real and negative. Cheng and Liu \cite{CL} constructed
representable matroids with nonunimodal Kazhdan--Lusztig polynomials,
so neither log-concavity nor real-rootedness holds for arbitrary
matroids.

A matroid of rank $d$ is \emph{sparse paving} if every $d$-element
subset of its ground set is either a basis or a circuit-hyperplane.
This class contains all uniform matroids. Mayhew, Newman, Welsh, and
Whittle \cite{MNWW} conjectured that asymptotically almost all matroids
are sparse paving.

Let $P_{m,d}(x)$ denote the
Kazhdan--Lusztig polynomial of the uniform matroid of rank $d$ on
$m+d$ elements. Proudfoot, Wakefield, and Young
\cite{PWY} computed $P_{1,d}(x)$ and interpreted its coefficients
using polygon dissections. Gedeon, Proudfoot, and Young
computed the equivariant polynomials of uniform matroids \cite{GPYeq}
and proved real-rootedness for corank one \cite[Theorem~3.3]{GPY}.
Gao, Lu, Xie, Yang, and
Zhang \cite{GLXYZ} obtained the positive formula for the coefficients that we use
below and proved real-rootedness for coranks $2\le m\le15$ and all
ranks using multiplier sequences and symbolic computations. Xie and Zhang \cite{XZ} proved
log-concavity for all uniform matroids. Gao, Li, Xie, Yang, and Zhang
\cite{GLiXYZ} recovered this log-concavity from induced log-concavity of
the equivariant Kazhdan--Lusztig polynomials.

Proudfoot, Xu, and Young \cite{PXY} introduced the $Z$-polynomial of
a matroid. Xie and Zhang \cite{XZz} proved its real-rootedness for all
sparse paving matroids and strict interlacing for uniform matroids
of fixed positive corank in consecutive ranks.

For sparse paving matroids, Lee, Nasr, and Radcliffe \cite{LNR}
gave a coefficient formula in terms of skew Young tableaux and proved
nonnegativity. Their formula shows that $P_M(x)$ depends only on the
rank, the corank, and the number of circuit-hyperplanes. Ferroni and
Vecchi \cite{FV} obtained this dependence by relaxing circuit-hyperplanes and computationally verified real-rootedness for sparse paving
matroids with at most $30$ elements. Ferroni, Nasr, and Vecchi \cite{FNV}
obtained formulas for the Kazhdan--Lusztig polynomials of paving matroids
by relaxing stressed hyperplanes.
Xie and Zhang~\cite{XZunimodal}
proved unimodality for loopless sparse paving matroids and
determined the modes.
Our main result is as follows.

\begin{theorem}\label{th:intro-main}
Every nonconstant Kazhdan--Lusztig polynomial of a sparse paving
matroid has only simple negative zeros.
\end{theorem}

By Newton's inequalities, Theorem~\ref{th:intro-main} implies
log-concavity of the coefficients.
Our proof uses a formula for relaxing circuit-hyperplanes.
Relaxing a circuit-hyperplane means declaring it to be a basis.
We use the convention that $P_M(x)=0$ when $M$ has a loop
\cite[Remark~2.8]{FV}. In particular, $P_{1,0}(x)=0$.
Write $\Delta_d(x)=P_{1,d}(x)-P_{1,d-1}(x)$.
Ferroni and Vecchi \cite[Corollary 3.8 and Theorem 5.2]{FV} showed that a sparse paving matroid
$M$ of rank $d\ge1$ and corank $m$, with $\lambda$ circuit-hyperplanes,
satisfies
\[
 P_M(x)=P_{m,d}(x)-\lambda\Delta_d(x).
\]
Set
\[
 \bar\lambda_{m,d}=\frac{\binom{m+d}{d}}{\max(m,d)+1}.
\]
The bound on the number of circuit-hyperplanes in \cite[Lemma 5.4]{FV} gives
$0\le\lambda\le\bar\lambda_{m,d}$.
For $d\ge5$, write $n=\lfloor(d-1)/2\rfloor$.
The following theorem applies to every real parameter satisfying this upper bound.

\begin{theorem}\label{th:interlacing}
Let $d\ge5$ and $m\ge2$ be integers, and let $\lambda$ be real with
$\lambda\le\bar\lambda_{m,d}$.
Then $P_{m,d}(x)-\lambda\Delta_d(x)$ has degree $n$ and only
simple negative zeros, and $P_{2,d-2}(x)$ strictly interlaces this polynomial.
\end{theorem}

For a fixed rank, Theorem~\ref{th:interlacing} gives a common strict
interlacer for all coranks $m\ge2$ and all real parameters
$\lambda\le\bar\lambda_{m,d}$. Section~\ref{sec:comparisons} establishes the
strict interlacing that we use in the proof and reduces the problem to
inequalities at the zeros of $P_{2,d-2}(x)$. The sign of the correction
term shows that it suffices to consider the largest allowed value of $\lambda$.
We state the required endpoint inequality in
Proposition~\ref{c:prop:endpoint}.

Section~\ref{sec:symbol} proves the endpoint inequality for $d\ge17$.
A partial fraction expansion and two residue comparisons reduce the
proof to three sign conditions. We first prove positivity of the
coefficient at the negative pole closest to zero. We combine the terms
at the positive poles into an integral representation with
a positive decreasing weight. The quotient of this sum by $P_{2,d-2}(x)$ has
positive residues at the zeros of $P_{2,d-2}(x)$. An integral lower bound
controls the derivative of the weight and supplies the comparison
needed in a recurrence in the corank. A direct estimate handles the
remaining small coranks.

Section~\ref{sec:lowranks} proves the endpoint inequality for
$5\le d\le16$. We take remainders modulo $P_{2,d-2}(x)$ and expand in powers
of $m-2$. The coefficients in this expansion are polynomials in $x$. We verify
their signs by exact computation on rational intervals containing the
zeros of $P_{2,d-2}(x)$. These signs give the required inequalities for every $m\ge2$.
We conclude with the proofs of the main theorems, treating ranks at most four
and coranks at most one separately.

\section{Reduction to the endpoint inequality}\label{sec:comparisons}

Throughout the rest of the paper, $d\ge5$ and $m\ge2$ are integers
unless we state otherwise. We reduce Theorem~\ref{th:interlacing} to a sign
inequality at the zeros of $P_{2,d-2}(x)$ when $\lambda=\bar\lambda_{m,d}$.
We first use residues and coefficient identities to prove that
$P_{1,d-2}(x)$ and $P_{2,d-2}(x)$ strictly interlace.
This interlacing shows that decreasing $\lambda$ preserves the required
signs at the zeros of $P_{2,d-2}(x)$. We then verify the degree and
leading coefficient needed to count all zeros and state the endpoint inequality.
Sections~\ref{sec:symbol} and~\ref{sec:lowranks} prove that inequality.

Let $p(x)$ and $q(x)$ be nonzero real polynomials of degrees $j$ and $k$,
respectively, with only simple real zeros. Write their zeros as
$\rho_1<\cdots<\rho_j$ and $\sigma_1<\cdots<\sigma_k$, respectively.
If $k=j+1$ and $j\ge1$, we write $p(x)\prec q(x)$ when
\[
 \sigma_1<\rho_1<\sigma_2<\cdots<\rho_j<\sigma_{j+1}.
\]
If $k=j\ge1$, we write $p(x)\prec q(x)$ when
\[
 \rho_1<\sigma_1<\cdots<\rho_j<\sigma_j.
\]
We say that $p(x)$ and $q(x)$ strictly interlace if
$p(x)\prec q(x)$ or $q(x)\prec p(x)$.
For $j=0$ and $k=1$, we take $p(x)\prec q(x)$ by convention.
We use the following characterization of strict interlacing in terms
of residues. See also \cite{HT}.

\begin{lemma}\label{c:lem:residue}
Let $p(x)$ and $q(x)$ be real polynomials with positive leading
coefficients. Suppose that $p(x)$ has degree $j\ge1$ and simple real
zeros $\rho_1<\cdots<\rho_j$.
\begin{enumerate}
\item[(i)] If $\deg q(x)=j$, then $p(x)\prec q(x)$ if and only if
\[
 \frac{q(\rho_i)}{p'(\rho_i)}<0\qquad\text{for every }i.
\]
\item[(ii)] If $\deg q(x)=j$, then $q(x)\prec p(x)$ if and only if
\[
 \frac{q(\rho_i)}{p'(\rho_i)}>0\qquad\text{for every }i.
\]
\item[(iii)] If $\deg q(x)=j-1$, then $q(x)\prec p(x)$ if and only if
\[
 \frac{q(\rho_i)}{p'(\rho_i)}>0\qquad\text{for every }i.
\]
\end{enumerate}
\end{lemma}
\begin{proof}
The residue of $q(x)/p(x)$ at $\rho_i$ is
\[
 \frac{q(\rho_i)}{p'(\rho_i)}.
\]
Since $p'(\rho_i)$ has sign $(-1)^{j-i}$, the stated interlacing
orders imply the corresponding residue signs.
Conversely, each prescribed set of residue signs gives a zero of
$q(x)$ between each pair of consecutive zeros of $p(x)$.
When $\deg q(x)=j$, we use the positive leading coefficient to obtain
$q(x)>0$ for sufficiently large positive $x$ and
$(-1)^j q(x)>0$ for sufficiently large negative $x$.
Negative residues give $q(\rho_j)<0$, so we obtain one further
zero to the right of $\rho_j$. Positive residues give
$(-1)^{j-1}q(\rho_1)>0$, so we obtain one to the left of $\rho_1$.
When $\deg q(x)=j-1$, the intervals between consecutive zeros of
$p(x)$ already account for all zeros of $q(x)$. By counting degrees in each
case, we obtain simple zeros and the asserted interlacing.
\end{proof}

Gedeon, Proudfoot, and Young \cite[Theorem~3.3]{GPY} proved
real-rootedness of $P_{1,d-2}(x)$. We use the coefficient identities
below to construct diagonal operators that give simple negative zeros for
$P_{1,d-2}(x)$ and $P_{2,d-2}(x)$ and preserve strict interlacing.

We have $P_{m,d}(0)=1$. By \cite[Theorem 1.3]{GLXYZ}, the coefficient
of $x^k$ in $P_{m,d}(x)$ is
\begin{equation}\label{c:eq:coeff}
 \frac{\binom{m+d}{k}}{d-k}
 \sum_{h=0}^{m-1}\binom{d-k+h}{h+k+1}\binom{k-1+h}{h},
 \quad1\le k\le n.
\end{equation}
In particular, $\deg P_{m,d}(x)=n$ and
$\deg P_{1,d-2}(x)=\deg P_{2,d-2}(x)=n-1$.
Specializing \eqref{c:eq:coeff} at coranks one and two gives
\begin{equation}\label{c:eq:Kidentity}
 P_{2,d-2}(x)=dP_{1,d-2}(x)-\Delta_d(x)/x.
\end{equation}

To prove that the diagonal operators below preserve simple real
zeros, we use the following consequence of
\cite[Theorem 3.1]{CC}.

\begin{lemma}\label{c:lem:strictmult}
Let $\gamma_0,\ldots,\gamma_{n-1}>0$, and define the linear map $T$ by
\[
 T\!\left(\sum_{k=0}^{n-1}p_kx^k\right)=
 \sum_{k=0}^{n-1}\gamma_kp_kx^k.
\]
If the Jensen polynomial
\[
 \sum_{k=0}^{n-1}\binom{n-1}{k}\gamma_kx^k
\]
has only negative real zeros, then $T$ sends nonzero real polynomials
of degree at most $n-1$ with simple real zeros to polynomials of the same
degree with simple real zeros. Moreover, if $p(x)$ and $q(x)$ have degree
at most $n-1$ and $p(x)\prec q(x)$, then $T(p(x))\prec T(q(x))$.
\end{lemma}
\begin{proof}
We apply Theorem~3.1 of \cite{CC} to see that $T$ preserves real-rootedness.
Since every $\gamma_k>0$, $T$ preserves degree and is a linear
homeomorphism on each $\mathbb R[x]_{\le j}$, $j\le n-1$.
A polynomial of degree $j$ with simple real zeros lies in the interior of
the set of real-rooted polynomials in $\mathbb R[x]_{\le j}$,
whereas one with a multiple zero lies on its boundary.
If $p(x)$ has degree $j$ and simple real zeros, we choose an open
neighborhood $U$ of $p(x)$ in $\mathbb R[x]_{\le j}$ consisting
of real-rooted polynomials. Then $T(U)$ is an open neighborhood
of $T(p(x))$ that contains only real-rooted polynomials. Since $T$ preserves
degree, $T(p(x))$ also has simple real zeros.

Let $p(x)$ and $q(x)$ strictly interlace. By Obreschkoff's theorem
\cite[Theorem 8]{Branden}, every nonzero real linear combination of
$p(x)$ and $q(x)$ has simple real zeros. We use linearity, injectivity, and
the preceding paragraph to obtain the same property for $T(p(x))$ and
$T(q(x))$. We apply the theorem once more to obtain strict interlacing.

If $p(x)$ and $q(x)$ have different degrees, we use degree
preservation to determine the interlacing order. Suppose they have the
same degree $j\ge1$ and $p(x)\prec q(x)$. By Vieta's formulas, $T$
multiplies each sum of zeros by $\gamma_{j-1}/\gamma_j>0$.
The sum of the zeros of $p(x)$ is smaller than that of $q(x)$.
Since $T(p(x))$ and $T(q(x))$ strictly interlace, we use this comparison
to conclude that $T(p(x))\prec T(q(x))$.

When the input polynomial has only negative zeros, $T$ preserves
coefficient signs and real-rootedness, so the image also has only negative zeros.
\end{proof}

We now describe the finite multiplier sequences that we will use.
If $p(x)=\sum_{k=0}^{n-1}p_kx^k$
has degree $n-1$, simple negative zeros, and positive leading coefficient,
then $p_k>0$ and the sequence $p_k/\binom{n-1}{k}$, $0\le k\le n-1$,
has Jensen polynomial $p(x)$.
We also use reciprocals of factorials as multipliers. For each integer
$s\ge0$, set
\[
 \mathscr L_s(x)=\sum_{k=0}^{n-1}\binom{n-1}{k}\frac{x^k}{(k+s)!}.
\]
We evaluate the generalized Laguerre polynomial of degree $n-1$ and
parameter $s$ at $-x$ to obtain a positive multiple of $\mathscr L_s(x)$.
Thus $\mathscr L_s(x)$ has simple negative zeros
\cite[Sections 18.2(vi) and 18.5(iii)]{DLMF}.
The sequences $1/(k+s)!$ and $1/(n+s-1-k)!$, indexed by $0\le k\le n-1$,
have Jensen polynomials $\mathscr L_s(x)$ and $x^{n-1}\mathscr L_s(1/x)$,
respectively. Lemma~\ref{c:lem:strictmult} therefore applies to the
corresponding diagonal operators. Compositions of these operators preserve
the conclusions of the lemma.

We write $\varthetax=x\,d/dx$ for the Euler operator. Rational functions
of $\varthetax$ act coefficientwise by substituting $k$ for $\varthetax$
on the monomial $x^k$. The denominators in these operators do not vanish
at the relevant coefficient indices. We apply an operator before evaluating
the resulting polynomial. We also write $(u)_j=u(u+1)\cdots(u+j-1)$, with $(u)_0=1$.

The preservation of simple real zeros by the following two operators is
a special case of Br\"and\'en~\cite[Corollary~3.5]{BrandenPF}.
We include a direct proof of the strict interlacing orders.

\begin{lemma}\label{c:lem:euler}
Let $f(x)$ be a nonconstant real polynomial with simple negative zeros
and positive leading coefficient. If $\alpha>0$ and $\beta>\deg f(x)$, then
$(\alpha+\varthetax)f(x)$ and $(\beta-\varthetax)f(x)$ have the same degree
as $f(x)$, have only simple negative zeros, and satisfy
\[
 (\beta-\varthetax)f(x)\prec f(x)\prec(\alpha+\varthetax)f(x).
\]
\end{lemma}
\begin{proof}
We use the zeros and positive leading coefficient of $f(x)$ to see that
all its coefficients are positive. The two operators multiply its coefficient of $x^k$ by
$\alpha+k$ and $\beta-k$, respectively, so both resulting polynomials have positive
coefficients and the same degree as $f(x)$. At a zero $r<0$ of $f(x)$,
their quotients by $f(x)$ have residues $r<0$ and $-r>0$, respectively.
We apply Lemma~\ref{c:lem:residue} to obtain the asserted strict interlacing and simple
real zeros. Since both polynomials have positive coefficients, we conclude that all their zeros are negative.
\end{proof}
For $\alpha>0$ and $\beta>n-1$, both operators
also satisfy the hypotheses of Lemma~\ref{c:lem:strictmult}.
Their Jensen polynomials of degree $n-1\ge1$ are
\[
 (1+x)^{n-2}\bigl(\alpha+(\alpha+n-1)x\bigr),\qquad
 (1+x)^{n-2}\bigl(\beta+(\beta-n+1)x\bigr).
\]
Thus we may compose these operators with the preceding diagonal operators.

Set
\begin{equation*}
 a_d=\frac{\sqrt{d^2+8}-d}{2},\qquad b_d=d+a_d.
\end{equation*}
Then $0<a_d<2/d<1$, $a_d b_d=2$, and $u(d-u)+2=(u+a_d)(b_d-u)$.
The zeros $-a_d$ and $b_d$ of this quadratic lie outside the interval
$[0,n-1]$.

For each integer $j\ge0$, let $U_j(z)$ denote the Chebyshev polynomial
of the second kind. It satisfies
\[
 U_j(\cos\theta)=\frac{\sin((j+1)\theta)}{\sin\theta},
 \qquad 0<\theta<\pi,
\]
by \cite[Eq.~(18.5.2)]{DLMF}. Set
\[
 f_d(x)=\sum_{k=0}^{n-1}\binom{d-k-3}{k}x^k.
\]
The coefficient expansion
\[
 U_{d-3}(z)=\sum_{k=0}^{n-1}(-1)^k
 \binom{d-k-3}{k}(2z)^{d-3-2k}
\]
gives
\begin{equation}\label{c:eq:Cheb}
 f_d(-1/(4z^2))=(2z)^{-(d-3)}U_{d-3}(z).
\end{equation}
The coefficient formula \eqref{c:eq:coeff} gives
\begin{equation}
 P_{1,d-2}(x)=(d-1)!\sum_{k=0}^{n-1}
 \frac{\binom{d-k-3}{k}}{(k+1)!(d-1-k)!}x^k,\label{c:eq:seed}
\end{equation}
and
\begin{equation}
 P_{2,d-2}(x)=d!\sum_{k=0}^{n-1}
 \frac{(k+a_d)(b_d-k)\binom{d-k-3}{k}}{(k+2)!(d-k)!}x^k.
 \label{c:eq:Kseed}
\end{equation}
By comparing the coefficients in \eqref{c:eq:seed} and~\eqref{c:eq:Kseed},
we find that, for $0\le k\le n-1$, the ratio of the coefficient of $x^k$ in
$P_{2,d-2}(x)$ to that in $P_{1,d-2}(x)$ is
\begin{equation}\label{c:eq:Kcoeff}
 \frac{d(k(d-k)+2)}{(k+2)(d-k)}.
\end{equation}

\begin{lemma}\label{c:lem:comparisonzeros}
For every integer $d\ge5$, the polynomials $f_d(x)$, $P_{1,d-2}(x)$, and $P_{2,d-2}(x)$
have $n-1$ simple negative zeros, and $P_{1,d-2}(x)\prec P_{2,d-2}(x)$.
In particular, each zero $r$ of $P_{2,d-2}(x)$ satisfies
\begin{equation}\label{c:eq:correctiondirection}
 \frac{\Delta_d(r)}{rP_{2,d-2}'(r)}=d\frac{P_{1,d-2}(r)}{P_{2,d-2}'(r)}>0.
\end{equation}
\end{lemma}
\begin{proof}
Using \eqref{c:eq:Cheb}, we obtain all $n-1$ zeros of $f_d(x)$ from
the positive zeros of $U_{d-3}(z)$. The resulting zeros of $f_d(x)$
are distinct and negative.
By \eqref{c:eq:Kseed}, we obtain
$(a_d+\varthetax)^{-1}P_{2,d-2}(x)$ by multiplying $f_d(x)$ by $d!$,
dividing its coefficient of $x^k$ by $(k+2)!(d-k)!$, and applying
$b_d-\varthetax$. We obtain $(b_d-\varthetax)^{-1}P_{2,d-2}(x)$ in
the same way, with $a_d+\varthetax$ as the final operator.
Lemma~\ref{c:lem:strictmult} applies because $a_d>0$ and $b_d>n-1$.
Both polynomials therefore have degree $n-1$, simple negative zeros,
and positive leading coefficients.
We apply Lemma~\ref{c:lem:euler} to these two polynomials to obtain
$n-1$ simple negative zeros for $P_{2,d-2}(x)$ and
\[
 (a_d+\varthetax)^{-1}P_{2,d-2}(x)\prec P_{2,d-2}(x)
 \prec(b_d-\varthetax)^{-1}P_{2,d-2}(x).
\]
Their quotients by $P_{2,d-2}(x)$ have positive and negative residues,
respectively.

By \eqref{c:eq:Kcoeff}, we have
\[
 P_{1,d-2}(x)=\frac1dP_{2,d-2}(x)+\frac{2}{d(a_d+b_d)}
 \left(\frac{b_d-1}{a_d+\varthetax}
       -\frac{a_d+1}{b_d-\varthetax}\right)P_{2,d-2}(x).
\]
Since $b_d-1>0$ and $a_d+1>0$, the displayed identity gives
\[
 \frac{P_{1,d-2}(r)}{P_{2,d-2}'(r)}>0
\]
at every zero $r$ of $P_{2,d-2}(x)$.
Both $P_{1,d-2}(x)$ and $P_{2,d-2}(x)$ have degree $n-1$ and positive
leading coefficients. Lemma~\ref{c:lem:residue}(ii) therefore gives
$P_{1,d-2}(x)\prec P_{2,d-2}(x)$.
This also gives $n-1$ simple negative zeros for $P_{1,d-2}(x)$.
Finally, we obtain the equality in \eqref{c:eq:correctiondirection}
from \eqref{c:eq:Kidentity}.
\end{proof}

To determine the degree, \eqref{c:eq:Kidentity} and~\eqref{c:eq:Kcoeff}
give the ratio of the coefficient of $x^{k+1}$ in $\Delta_d(x)$ to that
of $x^k$ in $P_{1,d-2}(x)$ as
\begin{equation}\label{c:eq:Deltacoeff}
 \frac{2d(d-k-1)}{(k+2)(d-k)},\qquad 0\le k\le n-1.
\end{equation}
In particular, $\Delta_d(x)/x$ has positive coefficients.
At the upper endpoint $\lambda=\bar\lambda_{m,d}$, set
\[
 F_{m,d}(x)=P_{m,d}(x)-\bar\lambda_{m,d}\Delta_d(x).
\]
\begin{proposition}\label{c:prop:leading}
For integers $d\ge5$ and $m\ge2$, the polynomial $F_{m,d}(x)$ has
degree $n$ and positive leading coefficient.
\end{proposition}
\begin{proof}
Both $P_{m,d}(x)$ and $\Delta_d(x)$ have degree at most $n$.
We prove the proposition by showing that
\[
 [x^n]F_{m,d}(x)
 =[x^n]P_{m,d}(x)-\bar\lambda_{m,d}[x^n]\Delta_d(x)>0.
\]
We use the coefficient estimates in the proofs of
\cite[Lemmas~21, 22, and~25]{LNR}, taking $i=n\ge2$ in their notation.
For $m=2$, we use
\[
 \bar\lambda_{2,d}=\frac{d+2}{2}
\]
as in the proof of Lemma~22.

For $n=2$ and $m\ge d$, we use
\[
 \bar\lambda_{m,d}=\frac{1}{m+1}\binom{m+d}{d}
\]
as in the first case of the proof of Lemma~25.

For $n\ge3$ and $m\ge3$, and for $n=2$ and $3\le m<d$,
we use
\[
 \bar\lambda_{m,d}\le\frac{2}{m+d+2}\binom{m+d}{d}
\]
as in the proof of Lemma~21 and the second case of the proof of
Lemma~25, respectively.

We combine these estimates with \eqref{c:eq:seed}
and~\eqref{c:eq:Deltacoeff} to obtain the desired positivity,
except possibly when $(m,d)=(3,5)$.
In that case, \eqref{c:eq:coeff} and~\eqref{c:eq:Deltacoeff} give
\[
 [x^2]F_{3,5}(x)=56-\frac{28}{3}\cdot5=\frac{28}{3}>0.
\]
\end{proof}

The preceding facts reduce Theorem~\ref{th:interlacing} to the following
inequality.

\begin{proposition}\label{c:prop:endpoint}
For integers $d\ge5$ and $m\ge2$, every zero $r$ of $P_{2,d-2}(x)$ satisfies
\begin{equation}\label{c:eq:endpointnodes}
 \frac{F_{m,d}(r)}{rP_{2,d-2}'(r)}>0.
\end{equation}
\end{proposition}

Equation~\eqref{c:eq:correctiondirection} shows that
\eqref{c:eq:endpointnodes} gives the required residue signs for every
$\lambda\le\bar\lambda_{m,d}$.
Proposition~\ref{c:prop:leading} and the positive coefficients of
$\Delta_d(x)/x$ provide the degree and leading coefficient needed to
apply Lemma~\ref{c:lem:residue}(ii) with $xP_{2,d-2}(x)$ as the denominator.

\section{The endpoint inequality for \texorpdfstring{$d\ge17$}{d >= 17}}\label{sec:symbol}
In this section, we prove Proposition~\ref{c:prop:endpoint} for $d\ge17$.
In Subsection~\ref{sec:criterion}, we use a partial fraction expansion and
two residue comparisons to reduce the proof to the three conditions of
Proposition~\ref{c:th:criterion}.
In Subsection~\ref{sec:leftcoeff}, we prove
Proposition~\ref{l:prop:smallpole}, which verifies condition~(\ref{c:cond:left}).
In Subsection~\ref{sec:kernel}, Lemma~\ref{g:lem:kernel} gives a quantitative
lower bound for $S_{m,d}(z)$ that verifies condition~(\ref{c:cond:kernel})
and supplies the estimate for $S_{m,d}(1)$ needed in the last subsection.
In Subsection~\ref{sec:positivity}, we combine this estimate with
Lemma~\ref{g:lem:tail} and a direct estimate for small coranks to prove
Proposition~\ref{g:th:gap}, which verifies condition~(\ref{c:cond:gap}).

Unless we state otherwise, $d\ge5$ and $m\ge2$ are integers.
\subsection{A criterion for the endpoint inequality}\label{sec:criterion}
The following expansion separates the contributions from the negative and
positive poles. Throughout this section, we use the coefficients that we
construct in the proof of the lemma.

\enlargethispage{\baselineskip}
\begin{lemma}\label{c:lem:decomposition}
For integers $d\ge5$ and $m\ge2$, there are real coefficients
$c_{m,d}^{\mathrm L}(p)$ and $c_{m,d}^{\mathrm R}(p)$ such that
\begin{equation}\label{c:eq:decomp}
 \begin{aligned}
 \frac{F_{m,d}(x)-P_{2,d-2}(x)}{x}
 ={}&\biggl(-4+
 \sum_{p\in\{a_d,3,\ldots,m+1\}}
 \frac{c_{m,d}^{\mathrm L}(p)}{p+\varthetax}\\
 &\quad+\sum_{p\in\{a_d,1,\ldots,m-1\}}
 \frac{c_{m,d}^{\mathrm R}(p)}{d+p-\varthetax}
 -\frac{6(d-3)(d-1)}{d(3d-7)(d-3-\varthetax)}\biggr)P_{2,d-2}(x).
 \end{aligned}
\end{equation}
Moreover, $c_{m,d}^{\mathrm L}(i)>0$ for $3\le i\le m$.
\end{lemma}
\pagebreak[3]
To control the alternating coefficients at the positive integer poles,
we combine the terms in \eqref{c:eq:decomp} with denominators
$b_d-\varthetax$ and $d+j-\varthetax$ into an integral. Set
\[
 g(t)=-\frac1{md\binom{m+d}{d}}
 \sum_{p\in\{a_d,1,\ldots,m-1\}}
 c_{m,d}^{\mathrm R}(p)e^{-(d+1+2p)t/2}.
\]
Their sum in \eqref{c:eq:decomp} is
\[
 -md\binom{m+d}{d}\int_0^\infty
 e^{-(d-1)t/2}g(t)P_{2,d-2}(e^tx)\,dt.
\]
When $g(t)$ is positive and nonincreasing, Lemma~\ref{c:lem:right} and the
negative factor give positive residues for the quotient of this sum
by $P_{2,d-2}(x)$.
Since $g(t)\to0$, it suffices to prove $-g'(t)>0$. Define
\begin{align*}
 E_{m,d}&=-g'(0),\\
 S_{m,d}(z)&=\frac1{2md\binom{m+d}{d}}\sum_{j=1}^{m-1}
 c_{m,d}^{\mathrm R}(j)(d+1+2j)(j-a_d)z^{j-1}.
\end{align*}
Differentiating $g(t)$ gives
\begin{equation}\label{c:eq:weightderivative}
 -g'(t)=e^{-(a_d+b_d+1)t/2}
 \left(E_{m,d}+\int_{e^{-t}}^1z^{-a_d}S_{m,d}(z)\,dz\right).
\end{equation}
The following criterion reduces \eqref{c:eq:endpointnodes} to three
sign conditions.

\begin{proposition}\label{c:th:criterion}
Let $d\ge5$ and $m\ge2$ be integers. Suppose that
\begin{enumerate}
\renewcommand{\theenumi}{\roman{enumi}}
\renewcommand{\labelenumi}{(\theenumi)}
\item\label{c:cond:left} $c_{m,d}^{\mathrm L}(a_d)>0$,
\item\label{c:cond:kernel} $S_{m,d}(z)\ge0$ for every $z\in[0,1]$,
\item\label{c:cond:gap} $E_{m,d}>0$.
\end{enumerate}
Then every zero $r$ of $P_{2,d-2}(x)$ satisfies
\[
 \frac{F_{m,d}(r)}{rP_{2,d-2}'(r)}>0.
\]
\end{proposition}
We now prove the expansion and the two residue comparisons needed for
the criterion.

\begin{proof}[Proof of Lemma~\ref{c:lem:decomposition}]
For integers $d\ge5$ and $m\ge1$, define the rational function
\begin{equation}\label{c:eq:Rsymbol}
 R_{m,d}(u)=\frac{(m+d)!(d-2-u)}{(d-1)!}\sum_{h=0}^{m-1}
 \frac1{h!(u+h+1)(u+h+2)(d-u+h)_{m-h}}.
\end{equation}
We compare the coefficients in \eqref{c:eq:coeff} and \eqref{c:eq:seed}
with indices $k+1$ and $k$, respectively, to obtain
\[
 P_{m,d}(x)=1+xR_{m,d}(\varthetax)P_{1,d-2}(x).
\]

Set
\begin{equation*}
 \tau_d(u)=\frac{(u+2)(d-u)}{d(u(d-u)+2)},\qquad
 \nu_d(u)=\frac{(d-u)(d-3-2u)(d-4-2u)
 ((u+1)(d-u-1)+2)}
 {(u+1)(u+3)(d-3-u)(u(d-u)+2)}.
\end{equation*}
We use the coefficient ratio \eqref{c:eq:Kcoeff} and compare consecutive
coefficients in \eqref{c:eq:Kseed} to obtain
\[
 \begin{aligned}
 P_{1,d-2}(x)&=\tau_d(\varthetax)P_{2,d-2}(x),\\
 \frac{P_{2,d-2}(x)-1}x&=\nu_d(\varthetax)P_{2,d-2}(x).
 \end{aligned}
\]

The second identity includes the highest coefficient index $k=n-1$,
since one of the factors $d-3-2k$ and $d-4-2k$ vanishes.
We combine these identities with \eqref{c:eq:Kidentity} and
the representation of $P_{m,d}(x)$ above to obtain
\begin{equation}\label{c:eq:symbolmain}
 \frac{F_{m,d}(x)-P_{2,d-2}(x)}{x}=\left(\tau_d(\varthetax)(R_{m,d}(\varthetax)-d\bar\lambda_{m,d})+\bar\lambda_{m,d}-\nu_d(\varthetax)\right)P_{2,d-2}(x).
\end{equation}

We expand
\[
 \tau_d(u)(R_{m,d}(u)-d\bar\lambda_{m,d})
 +\bar\lambda_{m,d}-\nu_d(u)
\]
into partial fractions. We let $c_{m,d}^{\mathrm L}(p)$ denote the coefficient of $(u+p)^{-1}$
for $p\in\{a_d,3,\ldots,m+1\}$, and let $c_{m,d}^{\mathrm R}(p)$
denote the coefficient of $(d+p-u)^{-1}$
for $p\in\{a_d,1,\ldots,m-1\}$.
In particular, $p=a_d$ gives
the nonintegral poles $-a_d$ and $b_d=d+a_d$.

We first compute the residues at the negative integer poles.
For $1\le i\le m$, only the summands with $h=i-1$ and, when $i\ge2$,
$h=i-2$ in
\eqref{c:eq:Rsymbol} contribute at $u=-i$. Their residues sum to
\[
 \frac{(m+d)!(d+i-2)(d+i-1)(d+2i-3)!}
 {(d-1)!(i-1)!(d+m+i-1)!}>0.
\]
At $u=-m-1$, only $h=m-1$ contributes, with residue
\[
 -md\binom{m+d}{d}\frac{d+m-1}{d+2m}.
\]
At $u=-1$, both $\tau_d(u)R_{m,d}(u)$ and $\nu_d(u)$ have
residue $-(d+1)$, so their difference has no pole there.
The factor $\tau_d(u)$ cancels the pole at $-2$. For $i\ge3$, we have
\[
 \tau_d(-i)=\frac{(i-2)(d+i)}{d(i(d+i)-2)}>0.
\]
The only additional contribution at these integer poles comes from
$-\nu_d(u)$ at $u=-3$, with positive residue
\[
 \frac{(d+1)(d+2)(d+3)^2}{d(3d+7)}.
\]
At these integer poles, $c_{m,d}^{\mathrm L}(i)>0$ for $3\le i\le m$.
Among the coefficients at negative integer poles, only
$c_{m,d}^{\mathrm L}(m+1)$ can be negative.

For $0\le j\le m-1$, only the summands with $0\le h\le j$ in
\eqref{c:eq:Rsymbol} contribute at $u=d+j$.
We expand $(1-t)^j$ and integrate each term against $t^{d+j}(1-t)$
to express the coefficient of $(d+j-u)^{-1}$ as
\[
 \frac{(-1)^{j+1}(m+d)!(j+2)}{(d-1)!j!(m-1-j)!}
 \int_0^1t^{d+j}(1-t)^{j+1}\,dt.
\]
The factor $\tau_d(u)$ cancels the pole at $d$. For the remaining
positive integer poles, we evaluate the beta integral and multiply
by $\tau_d(d+j)$ to obtain
\begin{equation}\label{c:eq:rightinteger}
 c_{m,d}^{\mathrm R}(j)=(-1)^{j+1}
 \frac{(m+d)!j(j+1)(j+2)(d+j+2)(d+j)!}
 {d!(m-1-j)!(d+2j+2)!(j(d+j)-2)},
 \qquad 1\le j\le m-1.
\end{equation}
These coefficients alternate in sign.

We expand $\tau_d(u)$ into partial fractions as
\[
 \tau_d(u)=\frac1d+\frac{2(b_d-1)}{d(a_d+b_d)(u+a_d)}
 -\frac{2(a_d+1)}{d(a_d+b_d)(b_d-u)}.
\]
For the pole at $-a_d$, set
\[
 \rho_d=\frac{db_d(a_d+b_d-4)(a_d+b_d-1)}
 {2(3(a_d+b_d)-11)}.
\]
This constant is positive.
We compute the coefficient of $(u+a_d)^{-1}$ in $\nu_d(u)$ as
\[
 \frac{2(b_d-1)\rho_d}{d(a_d+b_d)}.
\]
We substitute this residue into \eqref{c:eq:symbolmain} to obtain
\begin{equation}\label{c:eq:leftstar}
 c_{m,d}^{\mathrm L}(a_d)=\frac{2(b_d-1)}{d(a_d+b_d)}
 \bigl(R_{m,d}(-a_d)-d\bar\lambda_{m,d}-\rho_d\bigr).
\end{equation}
For the pole at $b_d$, set
\[
 \begin{aligned}
 \phi_d&=\frac{(a_d+1)(a_d+b_d+1)}{d(a_d+b_d)},\qquad
 \eta_d=\frac{(a_d+b_d+1)^2(a_d+b_d+3)(a_d+b_d+4)}
 {b_d(a_d+b_d)(3(a_d+b_d)+11)}.
 \end{aligned}
\]
These constants are positive.
We compute the coefficient of $(b_d-u)^{-1}$ in $\nu_d(u)$ as
\[
 -\frac{2\eta_d}{(a_d+b_d+1)(b_d+1)},
\]
and we substitute this value into \eqref{c:eq:symbolmain} to obtain
\begin{equation}\label{c:eq:rightstar}
 c_{m,d}^{\mathrm R}(a_d)=-\frac{2}{a_d+b_d+1}
 \left(\phi_d\bigl(R_{m,d}(b_d)-d\bar\lambda_{m,d}\bigr)
 -\frac{\eta_d}{b_d+1}\right).
\end{equation}
The term with denominator $d-3-u$ comes from $-\nu_d(u)$ and has the
coefficient in \eqref{c:eq:decomp}.

The rational functions corresponding to the operators in
\eqref{c:eq:decomp} and \eqref{c:eq:symbolmain} have only simple poles.
They have the same residues at all possible poles, and both tend to
$-4$ as $u\to\infty$.
They therefore coincide. To obtain \eqref{c:eq:decomp}, we substitute
$\varthetax$ for $u$ and apply the operators to $P_{2,d-2}(x)$.
\end{proof}

To control the possibly negative coefficient at $-m-1$, we normalize
the polynomials $(p+\varthetax)^{-1}P_{2,d-2}(x)$ to have the same
leading coefficient as $P_{2,d-2}(x)$, then compare the residues of their
quotients by $P_{2,d-2}(x)$. For $p>0$, the required normalization factor
is $n+p-1$.

\begin{lemma}\label{c:lem:left}
Let $d\ge5$ be an integer. For each zero $r$ of $P_{2,d-2}(x)$, the sequence indexed by $p=a_d,3,4,\ldots$,
\[
 (n+p-1)\frac{\left.(p+\varthetax)^{-1}P_{2,d-2}(x)\right|_{x=r}}{P_{2,d-2}'(r)},
\]
is positive, strictly decreasing, and converges to zero.
\end{lemma}
\begin{proof}
For $n=2$, we write $P_{2,d-2}(x)=1-x/r$ and obtain $-r/p$ for
the displayed expression. Since $r<0$, we obtain the assertion
and may henceforth assume $n\ge3$.

For successive parameters $(p,q)=(a_d,3)$ or $(i,i+1)$ with $i\ge3$, we have
\[
 \frac{n+p-1}{u+p}-\frac{n+q-1}{u+q}
 =(q-p)\frac{n-1-u}{(u+p)(u+q)}.
\]
We apply Lemma~\ref{c:lem:residue} to reduce strict decrease to
\[
 (n-1-\varthetax)((p+\varthetax)(q+\varthetax))^{-1}P_{2,d-2}(x)
 \prec P_{2,d-2}(x).
\]
At each zero of $\mathscr L_{q-1}'(x)$, the identity
$(p+\varthetax)\mathscr L_{q-1}(x)=p\mathscr L_{q-1}(x)+x\mathscr L_{q-1}'(x)$
shows that $(p+\varthetax)\mathscr L_{q-1}(x)$ has the same sign as
$\mathscr L_{q-1}(x)$. We apply Rolle's theorem and compare leading coefficients
to obtain
\[
 \mathscr L_{q-1}'(x)\prec(p+\varthetax)\mathscr L_{q-1}(x).
\]
Define the linear operator $T$ on $\mathbb R[x]_{\le n-1}$ to be
diagonal in the monomial basis and to satisfy
\[
 T\bigl((p+\varthetax)\mathscr L_{q-1}(x)\bigr)=P_{2,d-2}(x).
\]
Both polynomials have positive coefficients in every degree from $0$
to $n-1$, so this identity uniquely determines $T$. The matrix of $T$
in the monomial basis has positive diagonal entries.
By \eqref{c:eq:Kseed}, the map that multiplies the coefficient of $x^k$
by $1/(d-k)!$ and the operator $b_d-\varthetax$ send $d!f_d(x)$
to the Jensen polynomial of $T$
when $(p,q)=(a_d,3)$. For $(p,q)=(i,i+1)$, we additionally apply
\[
 (a_d+\varthetax)\prod_{j=3}^{i-1}(j+\varthetax).
\]
These operations give simple negative zeros, so
Lemma~\ref{c:lem:strictmult} applies to $T$.
Using $(k+1)\binom{n-1}{k+1}=(n-1-k)\binom{n-1}{k}$,
we compare coefficients to obtain
\[
 T\bigl(\mathscr L_{q-1}'(x)\bigr)
 =(n-1-\varthetax)((p+\varthetax)(q+\varthetax))^{-1}P_{2,d-2}(x).
\]
We apply $T$ to obtain strict interlacing. The polynomial on the left has
degree $n-2$ and positive leading coefficient, which proves the
strict decrease of the sequence.
Finally, for $0\le k\le n-1$, we have
\[
 \frac{n+p-1}{k+p}\to1
 \qquad\text{as the integer }p\to\infty.
\]
Thus the terms tend to
\[
 \frac{P_{2,d-2}(r)}{P_{2,d-2}'(r)}=0.
\]
From strict decrease to zero, we also obtain their positivity.
\end{proof}

For the positive poles, we use the following comparison for a positive
nonincreasing weight.
\begin{lemma}\label{c:lem:right}
Let $d\ge5$ be an integer and $c>(d-2)/2$. Let
$g\in C^1([0,\infty))$ be positive and
nonincreasing. Then every zero $r$ of $P_{2,d-2}(x)$ satisfies
\[
 \frac{1}{P_{2,d-2}'(r)}\int_0^\infty e^{-ct}g(t)P_{2,d-2}(e^t r)\,dt<0.
\]
In particular, taking $g(t)=1$ gives the corresponding inequality for
$(c-\varthetax)^{-1}P_{2,d-2}(x)$.
\end{lemma}
\begin{proof}
We first prove the assertion with $f_d(x)$ in place of $P_{2,d-2}(x)$.
Since $c>\deg f_d(x)$ and $g$ is positive and bounded, we integrate each
term to obtain a polynomial of the same degree with positive coefficients.
Fix $z\in(0,1)$ with $U_{d-3}(z)=0$ and put $r=-1/(4z^2)$.
We substitute $u=ze^{-t/2}$ in \eqref{c:eq:Cheb} to obtain the weight
\[
 w(u)=u^{2c-d+2}g(2\log(z/u)),\qquad 0<u\le z.
\]
Since $2c-d+2>0$ and $0<g(t)\le g(0)$, we have
$\lim_{u\to0^+}w(u)=0$.
Differentiating $w(u)$, we obtain
\[
 w'(u)=u^{2c-d+1}\bigl((2c-d+2)g(2\log(z/u))-2g'(2\log(z/u))\bigr)>0.
\]
Let $T_{d-2}(z)$ be the Chebyshev polynomial of the first kind, so
$T_{d-2}(\cos\theta)=\cos((d-2)\theta)$ for $0<\theta<\pi$
\cite[Eq.~(18.5.1)]{DLMF}. At the zero we fixed, $T_{d-2}(z)=\pm1$.
Using $T_{d-2}'(z)=(d-2)U_{d-3}(z)$ \cite[Eq.~(18.9.21)]{DLMF},
we integrate by parts to obtain
\[
 T_{d-2}(z)\int_0^z w(u)U_{d-3}(u)\,du
 =\frac1{d-2}\int_0^z w'(u)
       (1-T_{d-2}(z)T_{d-2}(u))\,du>0.
\]
The integrand is nonnegative because $|T_{d-2}(u)|\le1$.
Since $w'(u)>0$ and $T_{d-2}(u)$ is nonconstant, the integrand is
positive on a subinterval. The improper integrals converge since
$\int_0^z w'(u)\,du=w(z)<\infty$.
Reversing the substitution, we find that the
integral with $f_d(x)$ in place of $P_{2,d-2}(x)$ has the sign of $T_{d-2}(z)$
at $r$. At the zero we fixed, we have
\[
 U_{d-3}'(z)=-\frac{(d-2)T_{d-2}(z)}{1-z^2}
\]
by \cite[Eq.~(18.9.22)]{DLMF}. Differentiating \eqref{c:eq:Cheb},
we see that $f_d'(r)$ has the opposite sign.
Lemma~\ref{c:lem:residue} then gives
\[
 f_d(x)\prec\int_0^\infty e^{-ct}g(t)f_d(e^t x)\,dt.
\]
We apply the diagonal operator in \eqref{c:eq:Kseed}. It commutes with
the substitution $x\mapsto e^t x$ and with integration, and it preserves
strict interlacing by Lemma~\ref{c:lem:strictmult}.
We then apply Lemma~\ref{c:lem:residue}(i) to obtain
the required signs for $P_{2,d-2}(x)$. When $g(t)=1$, we integrate each monomial to identify
the integral with $(c-\varthetax)^{-1}P_{2,d-2}(x)$.
\end{proof}

\begin{proof}[Proof of Proposition~\ref{c:th:criterion}]
Using $b_d=d+a_d$, we integrate the monomials in $S_{m,d}(z)$ to obtain
\eqref{c:eq:weightderivative}.
Conditions~(\ref{c:cond:kernel}) and~(\ref{c:cond:gap}), together with
\eqref{c:eq:weightderivative}, give $-g'(t)>0$ for every $t\ge0$.
Since $g(t)\to0$, the weight is positive.
We apply Lemma~\ref{c:lem:right} with
$c=(d-1)/2$ to the integral involving $g(t)$, and with $c=d-3$ and $g(t)=1$ to
the term at the pole $d-3$. Both values of $c$ exceed $(d-2)/2$.
The sum of the terms corresponding to the positive poles in \eqref{c:eq:decomp},
including the term at $d-3$, therefore has a negative leading
coefficient and, after division by $P_{2,d-2}(x)$, positive residues.

Since $F_{m,d}(0)=P_{2,d-2}(0)=1$, the left-hand side of
\eqref{c:eq:decomp} is a polynomial. Proposition~\ref{c:prop:leading}
shows that it has degree $n-1$ and positive leading coefficient.
Since the sum corresponding to the positive poles has negative leading coefficient,
we compare leading coefficients in \eqref{c:eq:decomp} to obtain
\[
 \frac{c_{m,d}^{\mathrm L}(a_d)}{n+a_d-1}
 +\sum_{i=3}^{m+1}\frac{c_{m,d}^{\mathrm L}(i)}{n+i-1}>4.
\]
Fix a zero $r$ of $P_{2,d-2}(x)$. The sequence in Lemma~\ref{c:lem:left}
attains its positive minimum over $p\in\{a_d,3,\ldots,m+1\}$ at $p=m+1$.
Condition~(\ref{c:cond:left}) and Lemma~\ref{c:lem:decomposition} show that
all other coefficients $c_{m,d}^{\mathrm L}(p)$ are positive.
We replace the values for $p<m+1$ by this minimum and leave the term
with $p=m+1$ unchanged. This gives a lower bound equal to the minimum
times the sum in the preceding inequality. Hence the sum corresponding
to the negative poles has positive residues after division by $P_{2,d-2}(x)$.
Since $-4P_{2,d-2}(r)=0$, we combine the two contributions to obtain
\eqref{c:eq:endpointnodes}.
\end{proof}

\subsection{Positivity of the coefficient at \texorpdfstring{$-a_d$}{-a}}\label{sec:leftcoeff}
To verify condition~(\ref{c:cond:left}) of Proposition~\ref{c:th:criterion},
we construct a lower bound for
$R_{m,d}(-a_d)-d\bar\lambda_{m,d}$.
A recurrence for $R_{m,d}(u)$ shows that this lower bound increases with $m$,
so it suffices to show that its value at $m=2$ exceeds $\rho_d$.
\begin{proposition}\label{l:prop:smallpole}
For integers $d\ge9$ and $m\ge2$, we have
$c_{m,d}^{\mathrm L}(a_d)>0$.
\end{proposition}
\begin{proof}
For every integer $d\ge5$, we split off the last summand in
\eqref{c:eq:Rsymbol} to obtain the recurrence
\begin{equation}\label{c:eq:Rrec}
 (m+d-u-1)R_{m,d}(u)=(m+d)R_{m-1,d}(u)
 +md\binom{m+d}{d}\frac{d-2-u}{(m+u)(m+u+1)},\qquad m\ge2.
\end{equation}
For $m\ge2$, we set
\[
 G_m=R_{m,d}(-a_d)-d\binom{m+d}{d}
 \left(\frac1{m+1}+\frac2{m(m-1)}
 \left(\frac1{d+1}-\frac1{m+1}\right)\right).
\]
Since
\[
 0<\frac{2}{m(m-1)}\le1,
\]
the parenthesized expression is a convex
combination of $1/(m+1)$ and $1/(d+1)$. Hence
\[
 G_m\le R_{m,d}(-a_d)-d\bar\lambda_{m,d}.
\]
By \eqref{c:eq:leftstar}, it suffices to prove
\[
 G_m\ge G_2>\rho_d.
\]
At $m=2$, we substitute $a_d=2/b_d$ and $d=b_d-2/b_d$ to obtain
\[
 G_2-\rho_d=
 \frac{(b_d^2-2)\bigl(3b_d^4(b_d-4)+7b_d^3+16b_d^2+6(b_d-2)\bigr)}
 {2b_d^2(b_d-3)(b_d+1)(3b_d-2)}>0.
\]
For $m\ge3$, we substitute the definition of $G_m$ into
\eqref{c:eq:Rrec} and use $a_db_d=2$ to obtain
\begin{equation}\label{l:eq:majorantrec}
\begin{split}
 &\frac{m(m+1)\bigl((m+b_d-1)G_m-(m+d)G_{m-1}\bigr)}
 {2d\binom{m+d}{d}}\\
 &\quad=
 \frac{(1-2a_d)m^2+a_d^2(m-1)+a_d}{(m-a_d)(m+1-a_d)}
 \\
 &\qquad+\frac{(m-3)\bigl((4-a_d)(m+3)+(d-6)(d+1)+a_d(d+2)\bigr)}
 {(d+1)(m-2)(m-1)}\\
 &\qquad+\frac{(d-9)(d-2)+a_d(d-3)+2}{(d+1)(m-2)(m-1)}>0.
\end{split}
\end{equation}
The three terms on the right are nonnegative, and the first is positive,
since $d\ge9$, $m\ge3$, and $0<a_d<2/d<1/2$.
Finally, we use
\[
 \frac{m+d}{m+b_d-1}>1
\]
and \eqref{l:eq:majorantrec} to obtain, by induction on $m$,
$G_m>G_{m-1}>0$ for every $m\ge3$. Thus $G_m\ge G_2>\rho_d$.
\end{proof}

\subsection{A lower bound for \texorpdfstring{$S_{m,d}(z)$}{S(z)}}\label{sec:kernel}
The following lower bound verifies condition~(\ref{c:cond:kernel}) of
Proposition~\ref{c:th:criterion}. Its value at $z=1$ also supplies the
estimate needed in the recurrence argument for $E_{m,d}>0$ in
Subsection~\ref{sec:positivity}.

\begin{lemma}\label{g:lem:kernel}
For integers $d\ge10$ and $m\ge2$, we have
\begin{equation*}
 S_{m,d}(z)>
 \frac{d+3}{b_d+1}
 \frac{3(d-5)(d-3)(m-1)}{d(d+3+(m-2)z)^4},
 \qquad 0\le z\le1.
\end{equation*}
\end{lemma}

We first express $S_{m,d}(z)$ in terms of an auxiliary integral.
For integers $d\ge5$ and $m\ge1$, define
\[
 I_{m,d}(z)=\int_0^1t^d(1-t)(1-zt(1-t))^{m-1}\,dt.
\]
By expanding $I_{m,d}(z)$ and evaluating the beta integrals, we obtain
\begin{equation}\label{g:eq:rightcoeff}
 I_{m,d}(z)=\sum_{j=0}^{m-1}(-1)^j\binom{m-1}{j}
 \frac{(d+j)!(j+1)!}{(d+2j+2)!}z^j.
\end{equation}

For $m\ge2$, we use $\vartheta_z=z\,d/dz$ with the same coefficientwise convention
as for $\varthetax$ and set
\[
 \Psi(z)=\left(\frac{d+3}{2}+\vartheta_z\right)
 (3+\vartheta_z)(-I_{m,d}'(z)).
\]
Using \eqref{g:eq:rightcoeff} and~\eqref{c:eq:rightinteger}, together with
the factorization $j(d+j)-2=(j-a_d)(j+b_d)$, we obtain
\begin{equation}\label{g:eq:Sop}
 S_{m,d}(z)=\frac1d\left(\Psi(z)
 +(2-a_d)\int_0^1t^{b_d}\Psi(tz)\,dt\right).
\end{equation}

We use the following expansion and coefficient bound.
\begin{lemma}\label{g:lem:coefflower}
For each integer $d\ge10$, there are real coefficients $\zeta_{d,k}$,
$k\ge0$, such that for every integer $m\ge2$,
\begin{equation}\label{g:eq:Psibinomial}
 \Psi(z)=(m-1)\sum_{k=0}^{m-2}\binom{m-2}{k}
 \left(\frac z4\right)^k\left(1-\frac z4\right)^{m-2-k}\zeta_{d,k}.
\end{equation}
Moreover,
\[
 \zeta_{d,k}>\frac{3(d-5)(d-3)}{(4k+d+3)^4},\qquad k\ge0.
\]
\end{lemma}
\begin{proof}
We differentiate $I_{m,d}(z)$ and apply the binomial theorem to obtain
\[
 -\frac{I_{m,d}'(z)}{m-1}
 =\sum_{k=0}^{m-2}\binom{m-2}{k}
 \left(\frac z4\right)^k\left(1-\frac z4\right)^{m-2-k}v_k,
\]
where, for every integer $k\ge0$, we set
\[
\begin{aligned}
 v_k&=\int_0^1t^{d+1}(1-t)^2(2t-1)^{2k}\,dt\\
 &=2^{-d-3}\int_0^1s^{2k}(1-s^2)^2
 \sum_{h=0}^{n}\binom{d-1}{2h}s^{2h}\,ds.
\end{aligned}
\]
We obtain the second equality by pairing $t$ with $1-t$ and setting $s=2t-1$.
For real $\alpha$, we apply $\alpha+\vartheta_z$ to replace $v_k$ by
$(k+\alpha)v_k-kv_{k-1}$ and omit the second term at $k=0$.
We then apply the two operators that define $\Psi(z)$ to obtain
\eqref{g:eq:Psibinomial}, with coefficients independent of $m$.
We evaluate the beta integrals to obtain
\[
 \zeta_{d,0}=\frac{3}{(d+2)(d+4)},\qquad
 \zeta_{d,1}=\zeta_{d,0}\frac{d^2-7d-14}{(d+3)(d+6)}.
\]
For $k=0$, we use
\[
 \zeta_{d,0}>\frac{3}{(d+3)^2}
 >\frac{3(d-5)(d-3)}{(d+3)^4}.
\] For $k=1$, we combine the identity
\[
\begin{aligned}
 &(d^2-7d-14)(d+9)^2-(d-5)(d-3)(d+2)(d+3)\\
 &\qquad=2(d-10)(7d^2+50d+77)+316>0
\end{aligned}
\]
with the arithmetic--geometric mean inequality to obtain
\[
 \zeta_{d,1}>
 \frac{3(d-5)(d-3)}{(d+4)(d+6)(d+9)^2}
 >\frac{3(d-5)(d-3)}{(d+7)^4}.
\]

For $k\ge2$ and $0\le h\le n$, we substitute $t=s^2$ and start with
$w(t)=t^{h-1/2}(1-t)^2$. With $\alpha=3$ and then $\alpha=(d+3)/2$,
we integrate by parts using
\[
 \int_0^1\bigl((k+\alpha)t^k-kt^{k-1}\bigr)w(t)\,dt
 =\int_0^1t^k\left(\alpha w(t)+((1-t)w(t))'\right)dt.
\]
We apply both transformations to the weight to obtain
\[
 \frac{2h-1}{4}t^{h-5/2}(1-t)^3
 \bigl((2h-3)+(d-2h-2)t\bigr).
\]
All boundary terms vanish for $k\ge2$. At zero, their smallest
exponent is $k+h-3/2\ge1/2$. At one, they contain $(1-t)^3$ or
$(1-t)^4$. The final integrand is at worst of order $t^{-1/2}$
near zero, so the improper integrals converge. Set
\[
 \Omega_d(t)=\frac1d\sum_{j=0}^{\lfloor(d+1)/2\rfloor}
 (2j-3)(2j-1)(d-2j)\binom{d+1}{2j}t^j.
\]
The same formula defines $\Omega_d(t)$ for every integer $d\ge5$.
After summing over $h$ and collecting equal powers of $t$, we obtain
\begin{equation}\label{v8:eq:moment}
 \zeta_{d,k}=2^{-d-6}\int_0^1t^{k-5/2}(1-t)^3\Omega_d(t)\,dt,
 \qquad k\ge2.
\end{equation}
Using $\varthetax x^k=kx^k$, we evaluate $\Omega_d(1)$ as
\[
\begin{aligned}
 \Omega_d(1)
 &=\left.\frac{(\varthetax-3)(\varthetax-1)(d-\varthetax)}{2d}
 \bigl((1+x)^{d+1}+(1-x)^{d+1}\bigr)\right|_{x=1}\\
 &=(d-5)(d-3)2^{d-3}.
\end{aligned}
\]
For $1\le j\le\lfloor(d-1)/2\rfloor$, we use
\[
 \frac{\binom{d-1}{2j}}{\binom{d+1}{2j}}
 =\frac{(d+1-2j)(d-2j)}{d(d+1)},\qquad
 \frac{\binom{d-1}{2j-2}}{\binom{d+1}{2j}}
 =\frac{2j(2j-1)}{d(d+1)}
\]
to compare the coefficients of $t^j$ in the following recurrence. We check
the coefficients for $j=0$ and $j=\lfloor(d+1)/2\rfloor$ directly and obtain
\begin{equation}\label{v163:eq:Omega-rec}
 4(d-3)t\Omega_d'(t)
 =(d-2)(d+1)\bigl(\Omega_d(t)-(1-t)\Omega_{d-2}(t)\bigr),
 \qquad d\ge7.
\end{equation}
For $d\ge7$, put
\[
 \kappa_d=\frac{(d-2)(d+1)}{4(d-3)}\le\frac{d+3}4.
\]
Using \eqref{v163:eq:Omega-rec}, we obtain
\[
 \frac{d}{dt}\bigl(t^{-\kappa_d}\Omega_d(t)\bigr)
 =-\kappa_d t^{-\kappa_d-1}(1-t)\Omega_{d-2}(t).
\]
The initial polynomials $\Omega_5(t)=3(1-t)^3$ and
$\Omega_6(t)=35(t-1/5)^2+8/5$ are positive on $(0,1)$.
If $\Omega_{d-2}(t)>0$ on this interval, then the preceding derivative
is negative. Since $\Omega_d(1)>0$ for $d\ge7$, we proceed by induction
separately for even and odd $d$ to obtain
\[
 \Omega_d(t)>\Omega_d(1)t^{\kappa_d}
 \ge\Omega_d(1)t^{(d+3)/4}>0,\qquad 0<t<1.
\]
After substituting this bound into \eqref{v8:eq:moment}, we evaluate
the beta integral and apply the arithmetic--geometric mean inequality
to obtain
\[
 \zeta_{d,k}>
 \frac{3(d-5)(d-3)}{256(k+(d-3)/4)_4}
 >\frac{3(d-5)(d-3)}{(4k+d+3)^4},\qquad k\ge2.
\]
Combining this bound with the cases $k=0,1$, we obtain the asserted inequality.
\end{proof}

\begin{proof}[Proof of Lemma~\ref{g:lem:kernel}]
We interpret the weights in \eqref{g:eq:Psibinomial} as the probabilities of a
binomial distribution with parameters $m-2$ and $z/4$, whose mean is
$(m-2)z/4$. Using Lemma~\ref{g:lem:coefflower} and Jensen's inequality
for the convex function $u\mapsto(4u+d+3)^{-4}$ on $[0,\infty)$,
we obtain
\[
 \Psi(z)>
 \frac{3(d-5)(d-3)(m-1)}{(d+3+(m-2)z)^4}.
\]
The inequality is also strict when $m=2$ or $z=0$, since the bounds
on the coefficients are strict.
Since $tz\le z$ for $t,z\in[0,1]$, we substitute this bound into
\eqref{g:eq:Sop} and use
\[
 1+\frac{2-a_d}{b_d+1}=\frac{d+3}{b_d+1}
\]
to conclude the proof.
\end{proof}

\subsection{Positivity of \texorpdfstring{$E_{m,d}$}{E}}\label{sec:positivity}
The following proposition verifies condition~(\ref{c:cond:gap}) of
Proposition~\ref{c:th:criterion}.

\begin{proposition}\label{g:th:gap}
For integers $d\ge17$ and $m\ge2$, we have $E_{m,d}>0$.
\end{proposition}

We use a direct estimate for small coranks. For larger coranks,
we work backwards through a recurrence for $E_{m,d}$.

Differentiating $g(t)$ at zero and using \eqref{c:eq:rightstar}, we obtain
the following expression for $E_{m,d}$. We also use it to define $E_{1,d}$.
\begin{equation}\label{g:eq:gap}
 E_{m,d}
 =\frac{1}{md\binom{m+d}{d}}\left(
 \phi_d\bigl(R_{m,d}(b_d)-d\bar\lambda_{m,d}\bigr)
 -\frac{\eta_d}{b_d+1}
 -\frac12\sum_{j=1}^{m-1}c_{m,d}^{\mathrm R}(j)(d+1+2j)\right).
\end{equation}
For $2\le m\le d$, set
\begin{equation}\label{g:eq:B}
 B_{m,d}=\phi_d\left(\frac{a_d+1}{m(d+1)}
 -\frac{a_d+2}{(m+b_d)(m+b_d+1)}\right)
 +\frac{\eta_d}{md\binom{m+d}{d}}.
\end{equation}
For $m\ge d+1$, set
\begin{equation}\label{g:eq:Blarge}
 B_{m,d}=\phi_d\left(\frac{a_d+2}{m(m+1)}
 -\frac{a_d+2}{(m+b_d)(m+b_d+1)}\right)
 +\frac{\eta_d}{md\binom{m+d}{d}}.
\end{equation}
The next lemma reduces positivity for large coranks to a comparison
between $S_{m,d}(1)$ and $B_{m,d}$.

\begin{lemma}\label{g:lem:tail}
Fix an integer $d\ge5$ and an integer $m_0\ge2$.
If $S_{m,d}(1)>B_{m,d}$ for every $m\ge m_0$, then
$E_{m,d}>0$ for every $m\ge m_0-1$.
\end{lemma}
\begin{proof}
We first derive a recurrence for $E_{m,d}$. We substitute $u=b_d=d+a_d$
into \eqref{c:eq:Rrec} to obtain
\[
 (m-1-a_d)R_{m,d}(b_d)=(m+d)R_{m-1,d}(b_d)
 -\frac{md(a_d+2)\binom{m+d}{d}}{(m+b_d)(m+b_d+1)}.
\]
We set $c_{m-1,d}^{\mathrm R}(m-1)=0$ and use
\eqref{c:eq:rightinteger} to obtain
\[
 (m+d)c_{m-1,d}^{\mathrm R}(j)
 =(m-1-j)c_{m,d}^{\mathrm R}(j),\qquad1\le j\le m-1.
\]
The terms involving $\bar\lambda_{m,d}$ satisfy
\[
 \frac1{\max(m-1,d)+1}
 -\frac{m-1-a_d}{m(\max(m,d)+1)}
 =
 \begin{cases}
 (a_d+1)/(m(d+1)),&m\le d,\\
 (a_d+2)/(m(m+1)),&m\ge d+1.
 \end{cases}
\]
The second branch also holds at $m=d+1$, since
$\max(m-1,d)+1=m$.
We divide the first two identities by $md\binom{m+d}{d}$ and
substitute them, together with the preceding identity,
into \eqref{g:eq:gap}. Using
\[
 \binom{m+d-1}{d}=\frac{m}{m+d}\binom{m+d}{d},
\]
we obtain
\begin{equation}\label{g:eq:backward}
 (m-1-a_d)E_{m,d}=(m-1)E_{m-1,d}-S_{m,d}(1)+B_{m,d},
 \qquad m\ge2.
\end{equation}

We next obtain a lower bound for $E_{m,d}$ by retaining only
the possibly negative terms in \eqref{g:eq:gap}.
We substitute $u=b_d$ into the partial fraction expansion of
\[
 \frac{R_{m,d}(u)}{md\binom{m+d}{d}}.
\]
By the residue formulas in the proof of
Lemma~\ref{c:lem:decomposition}, the poles $-1,\ldots,-m$ contribute
positive terms, whereas the pole $-m-1$ contributes
\[
 -\frac{d+m-1}{(d+2m)(b_d+m+1)}.
\]
For $0\le j\le m-1$, the coefficient of $(d+j-u)^{-1}$ equals
$-(j+2)$ times the coefficient of $z^j$ in $I_{m,d}(z)$.
In \eqref{g:eq:gap}, we multiply these contributions from the positive poles by
$\phi_d$ and combine them with the alternating sum, using \eqref{c:eq:rightinteger},
\eqref{g:eq:rightcoeff}, and the identity
\[
\begin{aligned}
 &(j+2)\left(
 \frac{j(d+j+2)(d+1+2j)}{2d(j(d+j)-2)}
 -\frac{\phi_d}{j-a_d}\right)\\
 &\qquad=\frac1d\left(j(j-1)+\frac{d+11}{2}j+8-d\phi_d
 +\frac{(d-3+d\phi_d)(b_d-2)}{j+b_d}\right).
\end{aligned}
\]
This identity also holds at $j=0$. We sum these terms to express their
contribution to $E_{m,d}$ as
\[
\begin{aligned}
 &\frac{I_{m,d}''(1)}d+\frac{d+11}{2d}I_{m,d}'(1)
 +\left(\frac8d-\phi_d\right)I_{m,d}(1)\\
 &\qquad+\left(1-\frac3d+\phi_d\right)(b_d-2)
 \int_0^1z^{b_d-1}I_{m,d}(z)\,dz.
\end{aligned}
\]
The integral defining $I_{m,d}(z)$ gives $I_{m,d}(z)\ge0$ on $[0,1]$
and $I_{m,d}''(1)\ge0$. Moreover, $b_d>2$.
The definition of $\phi_d$, together with $a_d<2/d$ and $a_d+b_d>d$, gives
\begin{equation}\label{v3:eq:phibounds}
 \frac1d<\phi_d<\frac{(d+1)(d+2)}{d^3}.
\end{equation}
In particular, $d\phi_d<3$ for $d\ge5$.
In the displayed sum of four terms, only the term involving
$I_{m,d}'(1)$ can be negative.
For $0\le t\le1$, we apply the binomial theorem to obtain
\[
 \binom m2t^2(1-t)^2(1-t(1-t))^{m-2}\le1.
\]
We multiply by $2t^{d-1}/m$ and integrate to obtain
\[
 0\le-I_{m,d}'(1)
 =(m-1)\int_0^1t^{d+1}(1-t)^2(1-t(1-t))^{m-2}\,dt
 \le\frac2{md}.
\]

We discard the other three terms in that sum and the positive terms
from the negative poles.
Since $(d+m-1)/(d+2m)<1$, $b_d+m+1>m$, and $\binom{m+d}{d}\ge1$,
we obtain
\begin{equation}\label{g:eq:Elower}
 E_{m,d}\ge-\frac1m\left(
 \frac{d+11}{d^2}+2\phi_d+\frac{\eta_d}{d(b_d+1)}\right).
\end{equation}
Fix integers $k\ge m_0-1$ and $L>k$. By hypothesis,
$S_{m,d}(1)>B_{m,d}$ for $m\ge k+1$.
We iterate \eqref{g:eq:backward} and retain the term with $m=k+1$ to obtain
\[
 E_{k,d}\ge\left(\prod_{j=k}^{L-1}\left(1-\frac{a_d}{j}\right)\right)E_{L,d}
 +\frac{S_{k+1,d}(1)-B_{k+1,d}}{k}.
\]
Since $0<a_d<1$ and $k\ge1$, the product lies in $(0,1]$.
Using \eqref{g:eq:Elower} and letting $L\to\infty$, we conclude that
\[
 E_{k,d}\ge\frac{S_{k+1,d}(1)-B_{k+1,d}}{k}>0.
\]
\end{proof}

\begin{proof}[Proof of Proposition~\ref{g:th:gap}]
We prove $E_{m,d}>0$ directly for $2\le m\le(d-2)/3$ and then verify
the hypothesis of Lemma~\ref{g:lem:tail} for $m\ge\lceil(d+1)/3\rceil$.
The lemma then covers $m\ge\lfloor d/3\rfloor$, and the direct estimate
covers all remaining integers $m\ge2$.
We use \eqref{v3:eq:phibounds} throughout the proof and also need $\eta_d<d$.
In the defining formula for $\eta_d$, each numerator factor is less
than $d+5\le22d/17$, and the denominator exceeds $3d^3$. Hence
\[
 \eta_d<\frac13\left(\frac{22}{17}\right)^4d<d.
\]

First suppose $2\le m\le(d-2)/3$.
We use \eqref{g:eq:gap} to bound $E_{m,d}$ from below by estimating
the term containing $R_{m,d}(b_d)$ and the alternating sum separately.
We first prove by induction that
\[
 \frac{R_{m,d}(b_d)}{md\binom{m+d}{d}}>\frac1{d+2},\qquad m\ge2.
\]
At $m=2$, we evaluate the left-hand side as
\[
 \frac{b_d-1}{(b_d-2)(b_d+3)}>\frac1{d+2}.
\]
For $m\ge3$, we use the same normalization of \eqref{c:eq:Rrec} at
$u=b_d$ as in the proof of Lemma~\ref{g:lem:tail}.
Applying the induction hypothesis, we reduce the desired bound to
\[
 \frac{m-1}{d+2}-\frac{a_d+2}{(m+b_d)(m+b_d+1)}
 >\frac{m-1-a_d}{d+2}.
\]
We use $a_d(b_d+1)=a_d+2$ and
$(m+b_d)(m+b_d+1)>(b_d+1)(d+2)$ to verify this inequality.
Since $m-1-a_d>0$, we complete the induction.
Also,
\[
 \frac{\eta_d}{b_d+1}<1,\qquad
 m\binom{m+d}{d}\ge(d+1)(d+2).
\]

We retain the negative prefactor and write the last term in
\eqref{g:eq:gap} as an alternating sum.
For $m=2$, the sum consists of its $j=1$ term.
For $m\ge3$, we use \eqref{c:eq:rightinteger} to compare the absolute
values of consecutive summands. One positive factor decreases with
$j\ge1$, since
\[
 \frac{d}{du}\left(\frac{u(d+u+2)}{d(u(d+u)-2)}\right)
 =-\frac{2(u^2+2u+d+2)}{d(u^2+du-2)^2}<0,\qquad u\ge1.
\]
For $1\le j\le m-2$, we divide the product of the remaining factors
at $j+1$ by its value at $j$ to obtain
\[
 (m-1-j)\frac{j+3}{j+1}
 \frac{d+j+1}{(d+2j+4)(d+2j+1)}
 <\frac{2(m-2)}{d+6}<1.
\]
We therefore bound the alternating sum from below by its $j=1$ term.
We also use $\phi_d>1/d$ and
\[
 \frac1{d+2}-\frac1{m(d+1)}>0.
\]
We substitute these estimates into \eqref{g:eq:gap} and use
$\max(m,d)=d$ to obtain
\[
 E_{m,d}>\frac1{(d+1)(d+2)}-\frac1{d(d+1)m}
 -\frac{3(d+3)(m-1)}{d(d-1)(d+2)(d+4)}.
\]
The right-hand side is concave as a function of $m$.
At $m=2$ and $m=(d-2)/3$, respectively, it has the same sign as
\[
 d^3-5d^2-34d-10
 \qquad\text{and}\qquad d^2-11d-6.
\]
Both polynomials are increasing for $d\ge17$ and positive at $d=17$.
Thus the lower bound is positive throughout the interval, which proves
$E_{m,d}>0$ in the stated range.

To apply Lemma~\ref{g:lem:tail}, now suppose $m\ge\lceil(d+1)/3\rceil$.
Using $b_d<d+2/d$, we apply Lemma~\ref{g:lem:kernel} to obtain
\[
 \frac{(m+d+1)^4S_{m,d}(1)}{m-1}
 >\frac{3(d-5)(1-9/d^2)}{1+1/d+2/d^2}
 \ge\frac{360}{11}>\frac{129}{4},
\]
since each of the positive factors $d-5$, $1-9/d^2$, and
$(1+1/d+2/d^2)^{-1}$ increases with $d\ge17$.
After multiplying the two terms of $B_{m,d}$ by
\[
 \frac{(m+d+1)^4}{m-1},
\]
we bound the term containing $\eta_d$ by $1/4$ and the term containing
$\phi_d$ by $32$.

For the term containing $\eta_d$, put $s=m+d\ge23$.
Since $m\ge6$, $\eta_d<d$, and $\binom{s}{d}\ge\binom{s}{6}$,
the product formula for $\binom{s}{6}$ gives
\[
 \frac{\eta_d(s+1)^4}{m(m-1)d\binom{s}{d}}
 <\frac{24(s+1)^4}{(s-5)^6}
 \le\frac{24}{18^2}\left(\frac43\right)^4<\frac14.
\]

For the term containing $\phi_d$, we use \eqref{v3:eq:phibounds} and
$a_d<2/d$. First suppose $(d+1)/3\le m\le d$, and put
\[
 q=\frac{(m+d+1)^2}{m(d+1)}.
\]
Then
\[
 4\le q\le\frac{16}3,\qquad \frac{m}{m-1}\le\frac65.
\]
Since $a_d<1/2$ gives
$(m+b_d)(m+b_d+1)<(m+d+1)^2$, the $\phi_d$ contribution in
\eqref{g:eq:B} satisfies
\[
\begin{aligned}
 &\frac{(m+d+1)^4\phi_d}{m-1}\left(\frac{a_d+1}{m(d+1)}
       -\frac{a_d+2}{(m+b_d)(m+b_d+1)}\right)\\
 &\quad<\phi_d(d+1)\frac{m}{m-1}\,
          q\bigl((a_d+1)q-(a_d+2)\bigr)\\
 &\quad<\frac{32}{15}\left(1+\frac1d\right)^2
       \left(1+\frac2d\right)\left(10+\frac{26}d\right)\\
 &\quad\le\frac{32}{15}\left(\frac{18}{17}\right)^2
       \frac{19}{17}\frac{196}{17}<32.
\end{aligned}
\]
Here $q((a_d+1)q-(a_d+2))$ is positive and increasing in both
$q\ge4$ and $a_d>0$.

For $m\ge d+1$, we multiply the $\phi_d$ term in
\eqref{g:eq:Blarge} by the same factor and use the same denominator
comparison to bound it strictly by
\[
\begin{aligned}
 &\phi_d(a_d+2)\frac{(m+d+1)^2}{m(m+1)}
 \frac{(2d+1)m+(d+1)^2}{m-1}\\
 &\quad\le\frac{4\phi_d(a_d+2)(d+1)^2(3d+2)}{d(d+2)}\\
 &\quad<8\left(1+\frac1d\right)^4\left(3+\frac2d\right)
 \le8\left(\frac{18}{17}\right)^4\frac{53}{17}<32.
\end{aligned}
\]
The two factors involving $m$ decrease with $m$, so we evaluate them
at $m=d+1$, then apply \eqref{v3:eq:phibounds} and $a_d<2/d$.
The remaining factors decrease with $d$, so we evaluate them at $d=17$
to obtain the last bound.
In both cases, we therefore obtain
\[
 B_{m,d}<\frac{129(m-1)}{4(m+d+1)^4}<S_{m,d}(1).
\]

Lemma~\ref{g:lem:tail} now gives $E_{m,d}>0$ for
$m\ge\lfloor d/3\rfloor$ and completes the proof.
\end{proof}

\begin{proof}[Proof of Proposition~\ref{c:prop:endpoint} for $d\ge17$]
We apply Proposition~\ref{l:prop:smallpole},
Lemma~\ref{g:lem:kernel}, and Proposition~\ref{g:th:gap} to verify the
conditions~(\ref{c:cond:left}), (\ref{c:cond:kernel}), and~(\ref{c:cond:gap})
of Proposition~\ref{c:th:criterion}, respectively, and obtain
\eqref{c:eq:endpointnodes}.
\end{proof}

\section{Exact certificates for \texorpdfstring{$5\le d\le16$}{5 <= d <= 16}}\label{sec:lowranks}
We first prove Proposition~\ref{c:prop:endpoint} for $5\le d\le16$.
We check the signs of finitely many polynomials to treat every $m\ge2$.
We then prove Theorems~\ref{th:interlacing} and~\ref{th:intro-main}.

For integers $d\ge5$ and $m\ge2$, set
\begin{equation}\label{f:eq:majorantpoly}
 \widehat\lambda_{m,d}=\frac{(m+2)_{d-1}}{d!}
       -\frac{120(d-m)(m+5)_{d-4}}{(d+1)d!}.
\end{equation}
We define
\[
 H_{m,d}(x)=\frac{P_{m,d}(x)-\widehat\lambda_{m,d}\Delta_d(x)-P_{2,d-2}(x)}x.
\]
For fixed $d$, the quantity $\widehat\lambda_{m,d}$ is a polynomial in $m$.
We can also write
\[
 \widehat\lambda_{m,d}=\binom{m+d}{d}
 \left(\frac1{m+1}+\frac{120}{(m+2)_3}
 \left(\frac1{d+1}-\frac1{m+1}\right)\right).
\]
Since
\[
 0<\frac{120}{(m+2)_3}\le1,
\]
the parenthesized factor is a convex
combination of $1/(m+1)$ and $1/(d+1)$. Hence
$\widehat\lambda_{m,d}\ge\bar\lambda_{m,d}$, with equality if and only if
$m=2$ or $m=d$.
Subtracting $P_{2,d-2}(x)$ cancels the constant term without changing values
at its zeros. Thus $H_{m,d}(x)$ is a polynomial of degree at most $n-1$. For every zero
$r$ of $P_{2,d-2}(x)$, equation~\eqref{c:eq:correctiondirection} gives
\begin{equation}\label{f:eq:uppernodes}
 \frac{F_{m,d}(r)}{rP_{2,d-2}'(r)}\ge\frac{H_{m,d}(r)}{P_{2,d-2}'(r)}.
\end{equation}
Thus it suffices to determine the signs of the remainder of
$H_{m,d}(x)$ modulo $P_{2,d-2}(x)$.
\begin{lemma}\label{f:lem:polynomial}
For each integer $d\ge5$, the remainder of $H_{m,d}(x)$ modulo
$P_{2,d-2}(x)$ admits a polynomial extension in $m$ of degree at most $d-2$.
\end{lemma}
\begin{proof}
Fix $1\le k\le n$. We rewrite the binomial coefficients in
\eqref{c:eq:coeff} with lower arguments $d-2k-1$ and $k-1$ to express
the summand as a polynomial in $h$ of degree $d-k-2$, with leading coefficient
\[
 \frac1{(d-2k-1)!(k-1)!}.
\]
Summing over $0\le h<m$ raises the degree by one and divides the leading
coefficient by $d-k-1$. After multiplying by the prefactor in
\eqref{c:eq:coeff}, we find that $[x^k]P_{m,d}(x)$ is a polynomial
in $m$ of degree $d-1$, with leading coefficient
\[
 \frac{\binom{d-k-2}{k-1}}{k!(d-k)!}.
\]
Comparing with \eqref{c:eq:seed}, we find that the coefficient of
$m^{d-1}$ in the polynomial
\[
 \frac{P_{m,d}(x)-1}x
\]
is
\[
 \frac{P_{1,d-2}(x)}{(d-1)!}.
\]
The second term in \eqref{f:eq:majorantpoly} has degree $d-3$.
We therefore compute the coefficient of
$m^{d-1}$ in $H_{m,d}(x)$ as
\[
 \frac{dP_{1,d-2}(x)-\Delta_d(x)/x}{d!}=\frac{P_{2,d-2}(x)}{d!}.
\]
We reduce modulo $P_{2,d-2}(x)$ to eliminate this term and obtain the claimed degree bound.
\end{proof}

We expand this remainder in powers of $m-2$.
Lemma~\ref{f:lem:polynomial} shows that only powers with exponents from
$0$ to $d-2$ occur, and each power is nonnegative for $m\ge2$.
In view of \eqref{f:eq:uppernodes}, it therefore
suffices to show that all $d-1$ polynomials in this expansion have the same
sign as $P_{2,d-2}'(r)$ at each zero $r$ of $P_{2,d-2}(x)$.
We check these signs on rational intervals containing the zeros.
For $5\le d\le16$ and $0\le j\le d-2$, consider the coefficient of
$(m-2)^j$ in the remainder of $H_{m,d}(x)$ modulo $P_{2,d-2}(x)$. Let
$p_{d,j}(x)$ be the primitive integer polynomial that is a positive
rational multiple of this coefficient. All these polynomials are nonzero and
have degree less than $n-1$.
By Lemma~\ref{f:lem:polynomial}, we recover the remainder by Lagrange
interpolation from its values at $m=2,\ldots,d$.
Equations~\eqref{c:eq:coeff} and~\eqref{f:eq:majorantpoly} therefore
determine the polynomials $p_{d,j}(x)$ by exact rational arithmetic.

\begin{proof}[Proof of Proposition~\ref{c:prop:endpoint} for $5\le d\le16$]
We first locate the zeros of $P_{2,d-2}(x)$. Let $(A_i,q_i)$,
$1\le i\le n-1$, denote the pairs in Table~\ref{tab:intervals},
in the displayed order. They define disjoint intervals
\[
 \left(\frac{A_i}{q_i},\frac{A_i+1}{q_i}\right),
 \qquad 1\le i\le n-1.
\]
We substitute the endpoints into $P_{2,d-2}(x)$ to obtain
\begin{equation*}
 (-1)^{n-1-i}P_{2,d-2}(A_i/q_i)<0<
 (-1)^{n-1-i}P_{2,d-2}((A_i+1)/q_i).
\end{equation*}
Consequently the $i$th interval contains the $i$th zero $r_i$ of $P_{2,d-2}(x)$, and
$(-1)^{n-1-i}P_{2,d-2}'(r_i)>0$.

To check the signs of $p_{d,j}(x)$ on these intervals, we form the integer polynomials in $y$
\begin{equation}\label{f:eq:intervaltransform}
 (-1)^{n-1-i} q_i^{n-2}(1+y)^{n-2}
 p_{d,j}\!\left(\frac{A_i+(A_i+1)y}{q_i(1+y)}\right).
\end{equation}
The substitution in \eqref{f:eq:intervaltransform} maps $y>0$ onto
the $i$th interval. Every coefficient of this transformed polynomial is a positive
integer, so $(-1)^{n-1-i}p_{d,j}(x)>0$ throughout that interval.
The last column of Table~\ref{tab:intervals} gives the minimum of all
coefficients of these transformed polynomials over $0\le j\le d-2$ and $1\le i\le n-1$.
For example, when $d=7$ and $j=0$, we have $p_{7,0}(x)=163x+38$.
For the two intervals $(-2,-1)$ and $(-1/16,0)$,
\eqref{f:eq:intervaltransform} gives $288+125y$ and $445+608y$,
respectively. The supplementary material gives the explicit expansion of
\eqref{f:eq:intervaltransform} and includes a program for exact computation
and its output, with all primitive polynomials, transformed coefficients,
endpoint values, and minima in Table~\ref{tab:intervals}.

\begin{table}[htbp]
\centering
\small
\setlength{\tabcolsep}{5pt}
\caption{Rational intervals and minimum positive coefficients in
\eqref{f:eq:intervaltransform}. A pair $(A,q)$ denotes $(A/q,(A+1)/q)$.}
\label{tab:intervals}
\begin{tabular}{@{}r l r@{}}
\toprule
$d$ & $(A_1,q_1),\ldots,(A_{n-1},q_{n-1})$ & Least coefficient\\
\midrule
5 & $(-1,1)$ & $1$\\
6 & $(-1,1)$ & $1$\\
7 & $(-2,1),\ (-1,16)$ & $24$\\
8 & $(-4,8),\ (-1,32)$ & $117$\\
9 & $(-4,1),\ (-3,8),\ (-1,64)$ & $170$\\
10 & $(-2,1),\ (-2,8),\ (-1,64)$ & $4078$\\
11 & $(-6,1),\ (-2,2),\ (-5,32),\ (-1,128)$ & $672$\\
12 & $(-4,2),\ (-6,16),\ (-4,32),\ (-1,128)$ & $1723155$\\
13 & $(-10,1),\ (-4,4),\ (-5,16),\ (-3,32),\ (-1,256)$ & $447905$\\
14 & $(-3,1),\ (-2,2),\ (-4,16),\ (-5,64),\ (-1,256)$ & $329123$\\
15 & $(-14,1),\ (-2,1),\ (-4,8),\ (-6,32),\ (-4,64),\ (-3,1024)$ & $16890$\\
16 & $(-4,1),\ (-4,4),\ (-6,16),\ (-5,32),\ (-7,128),\ (-1,512)$ & $54329359$\\
\bottomrule
\end{tabular}
\end{table}

We evaluate the remainder at $r_i$ to recover $H_{m,d}(r_i)$.
For $0\le j\le d-2$, its coefficient of $(m-2)^j$ is a positive
multiple of $p_{d,j}(r_i)$. Thus every coefficient of
$(-1)^{n-1-i}H_{m,d}(r_i)$ is positive in these degrees, and all
higher coefficients vanish. In particular, the coefficient with $j=0$ is positive.
Consequently,
\[
 \frac{H_{m,d}(r_i)}{P_{2,d-2}'(r_i)}>0
 \qquad\text{for every }m\ge2.
\]
We now apply \eqref{f:eq:uppernodes} to conclude the proof.
\end{proof}

\FloatBarrier
We can now prove the main theorems.

\begin{proof}[Proof of Theorem~\ref{th:interlacing}]
Let $\lambda\le\bar\lambda_{m,d}$. We write
\[
 P_{m,d}(x)-\lambda\Delta_d(x)
 =F_{m,d}(x)+(\bar\lambda_{m,d}-\lambda)\Delta_d(x).
\]
Since $\bar\lambda_{m,d}-\lambda\ge0$,
Proposition~\ref{c:prop:leading} and \eqref{c:eq:Deltacoeff} show that
this polynomial has degree $n$ and positive leading coefficient.
Its constant term is $1$.
Proposition~\ref{c:prop:endpoint} and \eqref{c:eq:correctiondirection}
show that the rational function
\[
 \frac{P_{m,d}(x)-\lambda\Delta_d(x)}{xP_{2,d-2}(x)}
\]
has positive residues at every zero of $P_{2,d-2}(x)$. Its residue at zero is $1$ because $P_{2,d-2}(0)=1$.
We apply Lemma~\ref{c:lem:residue}(ii) to obtain
\[
 P_{m,d}(x)-\lambda\Delta_d(x)\prec xP_{2,d-2}(x).
\]
Since the largest zero of $xP_{2,d-2}(x)$ is zero, we conclude that
$P_{m,d}(x)-\lambda\Delta_d(x)$ has only simple negative zeros and that
$P_{2,d-2}(x)$ strictly interlaces this polynomial.
\end{proof}

\begin{proof}[Proof of Theorem~\ref{th:intro-main}]
Let $M$ be a sparse paving matroid of rank $d$ and corank $m$,
and let $\lambda$ be its number of circuit-hyperplanes. Suppose that
$P_M(x)$ is nonconstant. Then $M$ is loopless and $P_M(0)=1$.
For $d\le4$, the degree bound and
nonnegative coefficients \cite[Theorems 2.5 and 5.6]{FV} imply that
$P_M(x)$ is linear with positive coefficients. Its unique zero is simple
and negative.

We may therefore assume $d\ge5$. Corank zero gives the free matroid
and the constant polynomial $1$. When $m=1$, two distinct $d$-subsets
of the ground set meet in $d-1$ elements, whereas distinct
circuit-hyperplanes meet in at most $d-2$ elements. Hence $\lambda$
is either zero or one, and the relaxation formula gives $P_{1,d}(x)$
or $P_{1,d-1}(x)$. We apply Lemma~\ref{c:lem:comparisonzeros}
with parameter $d+2$ and $d+1$, respectively, to prove the assertion
in either case.

For $m\ge2$, the bound on $\lambda$ and the relaxation formula give
$P_M(x)=P_{m,d}(x)-\lambda\Delta_d(x)$ with
$0\le\lambda\le\bar\lambda_{m,d}$. We therefore apply Theorem~\ref{th:interlacing}
to complete the proof.
\end{proof}

\section*{Acknowledgements}
Philip B.~Zhang was supported by the National Natural Science Foundation of China under
Grant No.~12171362 and the Natural Science Foundation of Tianjin
Municipality under Grant No.~25JCYBJC00430.

\section*{Declaration of Generative AI}
During the preparation of this manuscript, the author used ChatGPT Pro
to explore proof strategies, write programs, and assist with writing.
The author takes full responsibility for the content of this manuscript.


\begin{thebibliography}{99}
\small

\bibitem{BHMPW}
T.~Braden, J.~Huh, J.~P.~Matherne, N.~Proudfoot, and B.~Wang,
\emph{Singular Hodge theory for combinatorial geometries},
J. Amer. Math. Soc., to appear,
\href{https://arxiv.org/abs/2010.06088}{arXiv:2010.06088}.

\bibitem{Branden}
P.~Br\"and\'en,
\emph{On operators on polynomials preserving real-rootedness and the
Neggers-Stanley conjecture},
J. Algebraic Combin. \textbf{20} (2004), no.~2, 119--130.

\bibitem{BrandenPF}
P.~Br\"and\'en,
\emph{On linear transformations preserving the P\'olya frequency property},
Trans. Amer. Math. Soc. \textbf{358} (2006), no.~8, 3697--3716.

\bibitem{CL}
R.~Cheng and S.~Liu,
\emph{Kazhdan--Lusztig polynomials of matroids need not be unimodal},
preprint (2026),
\href{https://arxiv.org/abs/2607.24186}{arXiv:2607.24186}.

\bibitem{CC}
T.~Craven and G.~Csordas,
\emph{Location of zeros. Part I: Real polynomials and entire functions},
Illinois J. Math. \textbf{27} (1983), no.~2, 244--278.

\bibitem{EPW}
B.~Elias, N.~Proudfoot, and M.~Wakefield,
\emph{The Kazhdan--Lusztig polynomial of a matroid},
Adv. Math. \textbf{299} (2016), 36--70.

\bibitem{FNV}
L.~Ferroni, G.~D.~Nasr, and L.~Vecchi,
\emph{Stressed hyperplanes and Kazhdan--Lusztig gamma-positivity for matroids},
Int. Math. Res. Not. IMRN (2023), no.~24, 20883--20942.

\bibitem{FV}
L.~Ferroni and L.~Vecchi,
\emph{Matroid relaxations and Kazhdan--Lusztig non-degeneracy},
Algebr. Comb. \textbf{5} (2022), no.~4, 745--769.

\bibitem{GLiXYZ}
A.~L.~L.~Gao, E.~Y.~H.~Li, M.~H.~Y.~Xie, A.~L.~B.~Yang, and Z.-X.~Zhang,
\emph{Induced log-concavity of equivariant matroid invariants},
J. Comb. Algebra (2026), published online first.

\bibitem{GLXYZ}
A.~L.~L.~Gao, L.~Lu, M.~H.~Y.~Xie, A.~L.~B.~Yang, and P.~B.~Zhang,
\emph{The Kazhdan--Lusztig polynomials of uniform matroids},
Adv. in Appl. Math. \textbf{122} (2021), 102117.

\bibitem{GPYeq}
K.~Gedeon, N.~Proudfoot, and B.~Young,
\emph{The equivariant Kazhdan--Lusztig polynomial of a matroid},
J. Combin. Theory Ser. A \textbf{150} (2017), 267--294.

\bibitem{GPY}
K.~Gedeon, N.~Proudfoot, and B.~Young,
\emph{Kazhdan--Lusztig polynomials of matroids:
a survey of results and conjectures},
S\'em. Lothar. Combin. \textbf{78B} (2017), Art.~80, 12~pp.

\bibitem{HT}
O.~Holtz and M.~Tyaglov,
\emph{Structured matrices, continued fractions, and root localization of polynomials},
SIAM Rev. \textbf{54} (2012), no.~3, 421--509.

\bibitem{LNR}
K.~Lee, G.~D.~Nasr, and J.~Radcliffe,
\emph{A combinatorial formula for Kazhdan--Lusztig polynomials of
sparse paving matroids},
Electron. J. Combin. \textbf{28} (2021), no.~4, Paper~P4.44.

\bibitem{MNWW}
D.~Mayhew, M.~Newman, D.~Welsh, and G.~Whittle,
\emph{On the asymptotic proportion of connected matroids},
European J. Combin. \textbf{32} (2011), no.~6, 882--890.

\bibitem{DLMF}
NIST Digital Library of Mathematical Functions,
\url{https://dlmf.nist.gov/}.

\bibitem{PWY}
N.~Proudfoot, M.~Wakefield, and B.~Young,
\emph{Intersection cohomology of the symmetric reciprocal plane},
J. Algebraic Combin. \textbf{43} (2016), no.~1, 129--138.

\bibitem{PXY}
N.~Proudfoot, Y.~Xu, and B.~Young,
\emph{The $Z$-polynomial of a matroid},
Electron. J. Combin. \textbf{25} (2018), no.~1, Paper~P1.26.

\bibitem{XZ}
M.~H.~Y.~Xie and P.~B.~Zhang,
\emph{The log-concavity of Kazhdan--Lusztig polynomials of uniform matroids},
J. Syst. Sci. Complex. \textbf{36} (2023), no.~1, 117--128.

\bibitem{XZz}
M.~H.~Y.~Xie and P.~B.~Zhang,
\emph{Real-rootedness of the $Z$-polynomials of sparse paving matroids
and strict interlacing for uniform matroids},
preprint (2026),
\href{https://arxiv.org/abs/2609.07636}{arXiv:2609.07636}.

\bibitem{XZunimodal}
M.~H.~Y.~Xie and Z.-R.~Zhang,
\emph{Unimodality of Kazhdan--Lusztig polynomials of sparse paving matroids},
preprint (2026),
\href{https://arxiv.org/abs/2609.40084}{arXiv:2609.40084}.

\end{thebibliography}
\end{document}